\documentclass[12pt,reqno]{article}
\usepackage{tikz}
\usetikzlibrary{positioning}
\usetikzlibrary{decorations.text}
\usetikzlibrary{decorations.pathmorphing}
\usetikzlibrary{matrix,arrows,decorations.pathmorphing}
\usepackage{hyperref}
\usepackage{empheq}
\usepackage{enumerate}
\usepackage{paralist}
\usepackage[doc]{optional}
\usepackage{xcolor}
\usepackage{float}
\usepackage{soul}
\usepackage{graphicx}
\definecolor{labelkey}{rgb}{0,0.08,0.45}
\definecolor{refkey}{rgb}{0,0.6,0.0}
\definecolor{Brown}{rgb}{0.45,0.0,0.05}
\definecolor{lime}{rgb}{0.00,0.8,0.0}
\definecolor{lblue}{rgb}{0.5,0.5,0.99}

 \usepackage{mathpazo}

\usepackage{amsmath}

\usepackage{stmaryrd}
\usepackage{amssymb}
\usepackage{amsthm}

\newcommand{\dir}{\ensuremath{\operatorname{dir}}}
\newcommand{\hzn}{\ensuremath{\operatorname{hzn}}}
\newcommand{\csm}{\ensuremath{\operatorname{csm}}}

\newcommand{\nnn}{\ensuremath{{n\in{\mathbb N}}}}

\newcommand{\menge}[2]{\big\{{#1}~\big |~{#2}\big\}}

\newcommand{\To}{\ensuremath{\rightrightarrows}}

\newcommand{\fenv}[1]%
{\ensuremath{\,\overrightarrow{\operatorname{env}}_{#1}}}
\newcommand{\benv}[1]%
{\ensuremath{\,\overleftarrow{\operatorname{env}}_{#1}}}

\newcommand{\scal}[2]{\left\langle{#1},{#2}  \right\rangle}

\newcommand{\RR}{\ensuremath{\mathbb R}}

\newcommand{\RPP}{\ensuremath{\mathbb{R}_{++}}}
\newcommand{\RMM}{\ensuremath{\mathbb{R}_{--}}}

\newcommand{\RX}{\ensuremath{\,\left]-\infty,+\infty\right]}}

\newcommand{\ran}{\ensuremath{\operatorname{ran}}}
\newcommand{\cran}{\ensuremath{\overline{\operatorname{ran}}\,}}

\newcommand{\Fix}{\ensuremath{\operatorname{Fix}}}
\newcommand{\Id}{\ensuremath{\operatorname{Id}}}

\newcommand{\minf}{\ensuremath{-\infty}}
\newcommand{\pinf}{\ensuremath{+\infty}}

{\begin{list}{}{%
\settowidth{\labelwidth}{\textrm{#1~}}%
\setlength{\leftmargin}{\labelwidth+\labelsep}}}
{\end{list}}

\usepackage{amsthm}
\usepackage[capitalize]{cleveref}
\crefname{equation}{}{equations}
\crefname{chapter}{Appendix}{chapters}
\crefname{item}{}{items}
\crefname{figure}{Figure}{figures}

\newtheorem{theorem}{Theorem}[section]
\newtheorem{lemma}[theorem]{Lemma}

\newtheorem{corollary}[theorem]{Corollary}

\newtheorem{example}[theorem]{Example}
\newtheorem{ex}[theorem]{Example}
\newtheorem{fact}[theorem]{Fact}

\providecommand{\norm}[1]{\lVert#1\rVert}

\providecommand{\innp}[1]{\langle#1\rangle}

\providecommand{\RR}{\mathbb{R}}

\providecommand{\ran}{\operatorname{ran}}
\providecommand{\cone}{\operatorname{cone}}
\providecommand{\ccone}{\overline{\operatorname{cone}}}

\providecommand{\Id}{\operatorname{{ Id}}}

\providecommand{\To}{\rightrightarrows}

\providecommand{\ran}{\operatorname{ran}}
\providecommand{\rec}{\operatorname{rec}}
\providecommand{\Id}{\operatorname{Id}}

\providecommand{\nnn}{{n\in\NN}}

\providecommand{\RR}{\mathbb{R}}

\definecolor{myblue}{rgb}{.8, .8, 1}
  \newcommand*\mybluebox[1]{%
    \colorbox{myblue}{\hspace{1em}#1\hspace{1em}}}

\allowdisplaybreaks 

\begin{document}
%
\title{\textsc
On a result of Pazy concerning the asymptotic behaviour of 
nonexpansive mappings}
\author{
Heinz H.\ Bauschke\thanks{
Mathematics, University
of British Columbia,
Kelowna, B.C.\ V1V~1V7, Canada. E-mail:
\texttt{heinz.bauschke@ubc.ca}.},~
Graeme R.\ Douglas\thanks{Computer Science,
University
of British Columbia,
Kelowna, B.C.\ V1V~1V7, Canada. E-mail: 
\texttt{graeme.r.doug@gmail.com.}}, 
~and Walaa M.\ Moursi\thanks{
Mathematics, University of
British Columbia,
Kelowna, B.C.\ V1V~1V7, Canada. E-mail:
\texttt{walaa.moursi@ubc.ca}.}}
\date{May 15, 2015}
\maketitle
\begin{abstract}
\noindent
In 1971, Pazy presented a beautiful trichotomy result concerning
the asymptotic behaviour of the iterates of a nonexpansive
mapping. In this note, we analyze the fixed-point free case in more
detail. Our results and examples give credence to the conjecture that the
iterates always converge cosmically. 
\end{abstract}
{\small

}
\section{Introduction}
Throughout, 
$X$ is a finite-dimensional real Hilbert space 
with inner product $\innp{\cdot,\cdot}$ and
induced norm $\norm{\cdot}$, and 
$T\colon X\to X$ is \emph{nonexpansive}, i.e., 
$(\forall x\in X)(\forall y\in X)$
$\|Tx-Ty\|\leq\|x-y\|$. 
Then, using \cite{Pazy}, the vector
\begin{empheq}[box=\mybluebox]{equation}
\label{e:v}
v := P_{\cran(\Id-T)}(0)
\end{empheq} 
is well defined. 
The following remarkable result\footnote{In fact, \cref{f:Pazy}
holds in general Hilbert space. See also  \cite{Reich73},
\cite{Reich81}, \cite{Reich82} and \cite{PlantReich} for even
more general settings. We thank Simeon Reich for bringing these
references to our attention.} was proved by A.~Pazy in 1971. 
\begin{fact}[\textbf{Pazy's trichotomy; \cite{Pazy}}]
\label{f:Pazy}
Let $x\in X$. Then 
\begin{equation}
\label{e:Pazy}
\lim_{n\to\infty} \frac{T^nx}{n} = -v.
\end{equation}
Moreover, exactly one of the following holds:
\begin{enumerate}
\item 
\label{f:Pazy1}
$0\in\ran(\Id-T)$,
and $(T^nx)_\nnn$ is bounded for every $x\in X$. 
\item
\label{f:Pazy2}
$0\in\cran(\Id-T)\smallsetminus\ran(\Id-T)$, 
$\|T^nx\|\to\infty$ and $\tfrac{1}{n}T^nx\to 0$ for every $x\in
X$. 
\item 
\label{f:Pazy3}
$0\notin\cran(\Id-T)$, and $\lim_{n\to\infty}
\tfrac{1}{n}\|T^nx\|>0$ for every $x\in X$. 
\end{enumerate}
\end{fact}
Now consider the case when $T$ does not have a fixed point. 
Let $x\in X$. 
In view of \cref{f:Pazy}, $\|T^nx\|\to \infty$ and
it is natural to ask whether additional asymptotic information is available
about the (eventually well defined) sequence
\begin{empheq}[box=\mybluebox]{equation}
\big(Q_n(x)\big)_\nnn := \left(\frac{T^nx}{\|T^nx\|}\right)_\nnn.
\end{empheq} 
Since $X$ is finite-dimensional, for every $x\in X$, $(Q_n(x))_\nnn$ has 
cluster points. 
If the sequence $(Q_n(x))_\nnn$ actually
converges, then we refer to this also as \emph{cosmic
convergence}\footnote{It will become clear in \cref{ss:cosmic} why we
speak of cosmic convergence.}. 
Combining \eqref{e:v} and \eqref{e:Pazy}, we obtain the following
necessary condition for cosmic convergence:
\begin{equation}
0\notin\cran(\Id-T)\quad\Rightarrow\quad v\neq 0
\;\text{and}\;
(\forall x\in X)\; Q_n(x) \to -v/\|v\|.
\end{equation}

\emph{The aim of this note is to provide conditions
sufficient for convergence of $(Q_n(x))_\nnn$ in the 
case when $0\in\cran(\Id-T)\smallsetminus\ran(\Id-T)$}. 

To the best of our knowledge, nothing was previously known about the
behaviour of $(Q_n(x))_\nnn$ in this case\footnote{Let us mention
in passing that the study of $(Q_n(x))_\nnn$ in the case when
$\Fix T\neq\varnothing$ seems of little interest. Indeed, if
$T\colon x\mapsto 0$, then $(Q_n(x))_\nnn$ is never well defined.}. 
The results in this note nurture the conjecture that 
the sequence $(Q_n(x))_\nnn$ actually converges.
Notation and notions not explicitly defined may 
be found in \cite{BC2011},
\cite{Rock70},  or \cite{Rock98}.

\section{Results}

\subsection{The one-dimensional case}

\begin{theorem}
\label{t:1D}
Suppose that $X$ is one-dimensional and that $\Fix
T=\varnothing$. Then $T$ admits cosmic convergence; in fact, 
exactly one of the following holds:
\begin{enumerate}
\item 
$(\forall x\in X)$ $Tx>x$, $T^nx\to\pinf$, and $Q_n(x)\to +1$.
\item
$(\forall x\in X)$ $Tx<x$, $T^nx\to\minf$, and $Q_n(x)\to -1$.
\end{enumerate}
\end{theorem}
\begin{proof}
We can and do assume that $X=\RR$. 
If there existed $a$ and $b$ in $\RR$ such that
$Ta>a$ and $Tb<b$, 
then the Intermediate Value Theorem would provide a
point $z$ between $a$ and $b$ such that $Tz=z$, which is absurd in
view of the hypothesis. 
It follows that
either $\ran(\Id-T)\subseteq \RMM$ or $\ran(\Id-T)\subseteq
\RPP$. 
Let us first assume that $\ran(\Id-T)\subseteq \RMM$, i.e., 
$(\forall x\in \RR)$ $x<Tx$. 
Let $x\in\RR$. 
On the one hand, we have
$x<Tx<T(Tx) = T^2x < T^3x<\cdots < T^nx < T^{n+1}x < \cdots$.
On the other hand, 
by \cref{f:Pazy}\ref{f:Pazy2}\&\ref{f:Pazy3}, 
$|T^nx|\to\pinf$. 
Altogether, $T^nx\to\pinf$ and hence $Q_n(x)\to +1$. 
Finally, the case when $\ran(\Id-T)\subseteq \RPP$ is treated
similarly. 
\end{proof}

\subsection{Composition of two projectors}

In this section, we assume that 
\begin{empheq}[box=\mybluebox]{equation}
\text{$A$ and $B$ are nonempty closed convex subsets of $X$}
\end{empheq} 
with corresponding projectors (nearest point mappings)
$P_A$ and $P_B$, respectively, 
and that 
\begin{empheq}[box=\mybluebox]{equation}
\text{$T = P_BP_A$.}
\end{empheq} 
We begin with a few technical lemmas. 

\begin{lemma}
\label{l:cutecone}
Let $K$ be a nonempty closed convex cone.
Then\footnote{We write $S^\ominus := \menge{x\in
X}{\sup\scal{x}{S}\leq 0}$ and $S^\oplus := -S^\ominus$ for a
subset $S$ of $X$.} $(K^\ominus)^\perp = K \cap (-K)$. 
\end{lemma}
\begin{proof}
We will use repeatedly the fact that (see \cite[Corollary~6.33]{BC2011}) 
$(K^\ominus)^\ominus = K$. 
``$\subseteq$'':
Indeed, $(K^\ominus)^\perp
\subseteq (K^\ominus)^\ominus = K$ and
$(K^\ominus)^\perp \subseteq (K^\ominus)^\oplus = -K$;
hence, $(K^\ominus)^\perp \subseteq K\cap (-K)$. 
``$\supseteq$'': Let $x\in K\cap (-K)$.
Then $\scal{x}{K^\ominus}\leq 0$ and $\scal{-x}{K^\ominus}\leq 0$
and thus $\scal{x}{K^\ominus}=0$, i.e., $x\in (K^\ominus)^\perp$. 
\end{proof}

\begin{lemma}
\label{l:0428}
The set of (oriented) functionals separating the sets
$A$ and $B$ satisfies
\begin{equation}
\label{e:0428a}
\mathcal{U} := 
\menge{u\in X\smallsetminus \{0\}}{\sup\scal{A}{u}\leq\inf\scal{B}{u}}
= \big(\ccone(A-B)\big)^\ominus\smallsetminus\{0\}. 
\end{equation}
Moreover\footnote{We use $\rec S := \menge{x\in X}{x+S\subseteq
S}$ to denote the recession cone of a nonempty convex
subset of $X$.}, 
\begin{equation}
\label{e:0428b}
(\rec A)\cap (\rec B)\subseteq \bigcap_{u\in\mathcal{U}}
\{u\}^\perp = \ccone(A-B)\cap\ccone(B-A).
\end{equation}
Consequently, if $A\cap B=\varnothing$, then
$\mathcal{U}\neq\varnothing$ and $(\rec A)\cap (\rec B)$ is
a nonempty closed convex cone that is
contained in a proper hyperplane of $X$. 
\end{lemma}
\begin{proof}
Since \eqref{e:0428a} is easily checked, we turn to
\eqref{e:0428b}: Let us first deal with the inclusion. 
If $\mathcal{U}=\varnothing$, then the intersection is trivially
equal to $X$ and we are done. 
So suppose that  $u\in\mathcal{U}$,
set $R :=\rec A$ and $S:=\rec B$.
Then $A+R=A$ and $B+S=B$; consequently,
$\sup\scal{A+R}{u}\leq\inf\scal{B+S}{u}$.
Since $R$ and $S$ are cones, we deduce that
$R\subseteq\{u\}^\ominus$ and $S\subseteq \{u\}^\oplus$. 
Therefore, $R\cap S\subseteq \{u\}^\perp$. 
This completes the proof of the inclusion.
Now $x\in \bigcap_{u\in\mathcal{U}}\{u\}^\perp$
$\Leftrightarrow$ 
$(\forall u\in \mathcal{U})$ $\scal{x}{u}=0$
$\Leftrightarrow$ 
$x \in \mathcal{U}^\perp = (\ccone(A-B))^{\ominus\perp} =
\ccone(A-B)\cap(-\ccone(A-B)) = \ccone(A-B)\cap\ccone(B-A)$ by
\cref{l:cutecone}. 
The ``Consequently'' part follows from the Separation Theorem 
(see, e.g., \cite[Theorem~2.5]{MN}).
\end{proof}

\begin{lemma}
\label{l:0427}
$0\in\cran(\Id-T)\smallsetminus\ran(\Id-T)$
$\Leftrightarrow$
$\Fix T=\varnothing$. 
\end{lemma}
\begin{proof}
By \cite{B03} (see also \cite{Victoria} for extensions to firmly
nonexpansive operators), we always have 
$0\in\cran(\Id-T)$, and this implies the result.
\end{proof}

We are now ready for the main result of this section.

\begin{theorem}
\label{t:0427}
Suppose that $\Fix T=\varnothing$. 
Let $b_0 := x \in X$ and set
$(\forall\nnn)$ 
$a_{n+1} := P_Ab_n$ and
$b_{n+1} := P_Ba_{n+1}=Tb_n$.
Then the following hold:
\begin{enumerate}
\item 
\label{t:0427i}
$\|a_n\|\to\pinf$,
$\|b_n\|\to\pinf$,
$b_n-a_n\to g$, and 
$a_{n+1}-b_n\to -g$,
where $g := P_{\overline{B-A}}(0)$. 
\item 
\label{t:0427ii}
All cluster points of $(b_n/\|b_n\|)_\nnn$ lie in
the set 
\begin{equation}
\big((\rec A)\cap(\rec B)\big) \cap
\big((\rec A)\cap(\rec B)\big)^\oplus,
\end{equation}
which is a closed convex cone in $X$ that properly contains
$\{0\}$. 
\item 
\label{t:0427iii}
Neither $(\rec A)\cap(\rec B)$ 
nor $((\rec A)\cap(\rec B))^\oplus$ 
is a linear subspace of $X$. 
\item \emph{\textbf{(cosmic convergence)}}
\label{t:0427iv}
The sequence $(Q_n(x))_\nnn = (b_n/\|b_n\|)_\nnn$ converges provided one of the
following holds:
\begin{enumerate}
\item
\label{t:0427iva}
$((\rec A)\cap(\rec B)) \cap
((\rec A)\cap(\rec B))^\oplus$ is a ray.
\item
\label{t:0427ivb}
$(\rec A)\cap(\rec B)$ is a ray. 
\item
\label{t:0427ivc}
$\dim X = 2$. 
\end{enumerate}
\end{enumerate}
\end{theorem}
\begin{proof}
Set $R := (\rec A)\cap(\rec B)$, which is a nonempty closed
convex cone. 

\ref{t:0427i}: See \cite[Theorem~4.8]{B94}.

\ref{t:0427ii}: 
Note that \ref{t:0427i} makes the quotient sequence eventually
well defined. 
Let $q$ be cluster point of $(b_n/\|b_n\|)_\nnn$, say
\begin{equation}
\frac{b_{k_n}}{\|b_{k_n}\|}\to q
\end{equation}
for some subsequence $(b_{k_n})_\nnn$ of $(b_n)_\nnn$. 
Then \cite[Proposition~6.50]{BC2011} implies that
$q\in\rec B$. 
Furthermore, since $a_{k_n}-b_{k_n}\to -g$ and
$\|b_{k_n}\|\to\pinf$, we deduce that 
\begin{equation}
\frac{a_{k_n}}{\|b_{k_n}\|} = \frac{a_{k_n}-b_{k_n}}{\|b_{k_n}\|}
+ \frac{b_{k_n}}{\|b_{k_n}\|} \to q. 
\end{equation}
As before, this implies that $q\in\rec A$. 
Thus
\begin{equation}
q\in (\rec A)\cap(\rec B)=R.
\end{equation}
On the other hand, 
using \cite[Theorem~3.1]{Zara} and
\cite[Proposition~6.34]{BC2011}, we have
\begin{subequations}
\begin{align}
b_{n+1}-b_n &= (b_{n+1}-a_{n+1})+(a_{n+1}-b_n)
\in \ran(P_B-\Id) + \ran(P_A-\Id)\\
& \subseteq \cran(P_B-\Id) + \cran(P_A-\Id)=(\rec B)^\oplus +
(\rec A)^\oplus\\
&\subseteq \overline{(\rec A)^\oplus +
(\rec B)^\oplus} 
= \big((\rec A)\cap(\rec B)\big)^\oplus = R^\oplus. 
\end{align}
\end{subequations}
It follows that
$b_n-b_0 = \sum_{k=0}^{n-1} (b_{k+1}-b_k) \in nR^\oplus =
R^\oplus$;
hence, $(b_n-b_0)/\|b_n-b_0\|\in R^\oplus$ which implies that
$q\in\ R^\oplus$. 
Altogether, $q\in R\cap R^\oplus$.
Since $\|q\|=1$, we deduce that $\{0\}\subsetneqq R\cap R^\oplus$. 
Finally, if $R$ was a linear subspace of $X$,
then $R\cap R^\oplus = R\cap R^\perp = \{0\}$, which is absurd. 
Hence $R$ is not a linear subspace.
If $R^\oplus$ were a linear subspace of $X$, then so would be 
$R^{\oplus\oplus}=R$, which is absurd. 

\ref{t:0427iv}:
In view of \ref{t:0427ii}, $R\cap R^\oplus$ contains a ray and it
suffices to show that $R\cap R^\oplus$ is precisely a ray. 
Indeed, each of the listed conditions guarantees that ---
for \ref{t:0427ivc} use \ref{t:0427iii}. 
\end{proof}

\begin{figure}[h!]
\begin{center}
\includegraphics[scale=0.5]{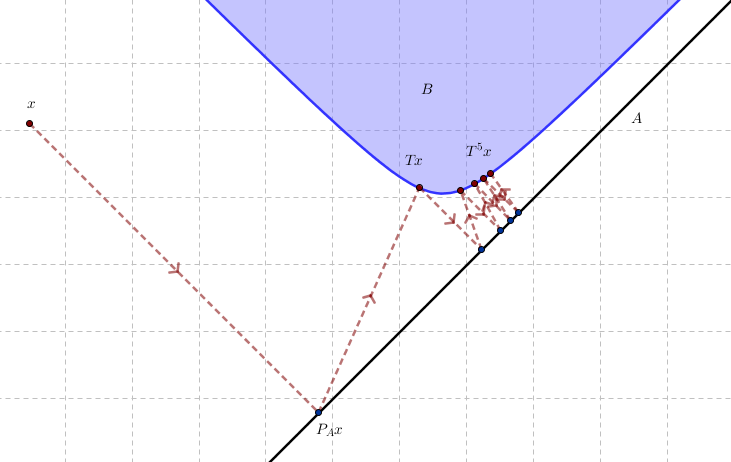}
\end{center}
\caption{A \texttt{GeoGebra} \cite{geogebra} snapshot 
in $\RR^2$ for two sets $A$ (the black line) and 
$B$ (the blue region) illustrating \cref{t:0427}\ref{t:0427ivc}.
Shown are the first few iterates 
of the sequence $(T^n x)_\nnn = (b_n)_\nnn$ (red points)
and of the sequence $(a_n)_\nnn$
(blue points).
We visually confirm cosmic convergence: 
the sequence $(Q_n(x))_\nnn$
converges to $(1/\sqrt{2})(1,1)$.
}
\label{figure:i}
\end{figure}

In \cref{figure:i}, we visualize 
\cref{t:0427}\ref{t:0427ivc} for the case 
when $A$ and $B$ are nonintersecting unbounded
closed convex subsets in the Euclidean plane.

\subsection{Firmly nonexpansive operators}

Recall that $x\in X$ belongs to the \emph{horizon cone} of a 
nonempty subset $C$ of $X$, written $x\in C^\infty$ if 
there exist sequences $(c_n)_\nnn$ in $C$ and
$(\alpha_n)_\nnn$ in $\RPP$ such that $\lambda_n\to 0$ and
$\lambda_n c_n\to x$.
Note that $C^\infty= \overline{C}^\infty$; furthermore,
if $C$ is closed and convex, then $C^\infty=\rec C$ 
(see \cite[Section~6.G]{Rock98}). 
The notion of the horizon cone allows us to present a superset of
cluster points of the iterates of $T$. 

\begin{theorem}
\label{t:T}
Suppose that $\Fix T = \varnothing$,
let $x_0 := x \in X$, and set $(\forall\nnn)$ $x_{n+1}:=Tx_n$. 
Then the following hold:
\begin{enumerate}
\item
\label{t:Ti}
All cluster points of $(x_n/\|x_n\|)_\nnn$ lie in the cone
\begin{equation}
R := \big(\ran T\big)^\infty \cap \big(\ran(T-\Id)\big)^\infty = 
\big(\ran T\big)^\infty \cap \rec\big(\cran(T-\Id)\big). 
\end{equation}
\item
\label{t:Tii}
If $T$ is firmly nonexpansive, then 
$R = \rec(\cran T)\cap\rec(\cran(T-\Id))$. 
\item \textbf{\emph{(cosmic convergence)}}
\label{t:Tiii}
If $R$ is a ray, then $(Q_n(x))_\nnn = (x_n/\|x_n\|)_\nnn$
converges. 
\end{enumerate}
\end{theorem}
\begin{proof}
By \cref{f:Pazy}, $\|x_n\|\to\infty$; thus, the quotient sequence
is eventually well defined. 
\ref{t:Ti}: 
Let $q$ be a cluster point of  $(x_n/\|x_n\|)_\nnn$.
It is clear that $q\in (\ran T)^\infty$. 
For every $\nnn$, we have
$x_{n+1}-x_0 = \sum_{k=0}^n (x_{k+1}-x_k) \in
(n+1)\ran(T-\Id) \subseteq \cone\ran(T-\Id)$;
hence, $(x_{n+1}-x_0)/\|x_{n+1}\|\in\cone\ran(T-\Id)$
and thus $q\in(\ran(T-\Id))^\infty$. 
Since $T$ is nonexpansive, $\cran(\Id-T)$ is convex (see
\cite[Lemma~4]{Pazy} which yields the right identity. 
\ref{t:Tii}:
Since $T$ is firmly nonexpansive, so is $\Id-T$ which implies
that $\cran(\Id-(\Id-T))=\cran T$ is convex (again by
\cite[Lemma~4]{Pazy}). 
The conclusion now follows because the horizon cone and recession
cone coincide for closed convex sets. 
\ref{t:Tiii}: This is clear. 
\end{proof}

The following result allows a reduction to lower-dimensional
cases. 

\begin{theorem}
\label{t:extend}
Let $Y$ be a linear subspace of $X$, and let $B\colon Y\To Y$ be
maximally monotone. Set $A := BP_Y$ and suppose that 
$T=J_A := (\Id+A)^{-1}$. 
Let $x\in X$.
Then the following hold:
\begin{enumerate}
\item 
\label{t:extendi}
$A\colon X\To X$ is maximally monotone and 
$T = P_{Y^\perp} + J_BP_Y$, where $J_B := (\Id+B)^{-1}$.
\item
\label{t:extendii}
$(\forall\nnn)$ $T^nx = P_{Y^\perp}x + J_B^n(P_Yx)$. 
\end{enumerate}
\end{theorem}
\begin{proof}
\ref{t:extendi}:
This follow from \cite[Proposition~23.23]{BC2011}.
\ref{t:extendii}: Clear from \ref{t:extendi} and induction. 
\end{proof}

We are now in a position to obtain a positive result for
proximity operators of certain convex functions. 

\begin{corollary}
\label{c:extend}
Let $f\colon X\to\RX$ be convex, lower semicontinuous, and proper
on $\RR$, let $a\in X$ such that $\|a\|=1$, set
$F\colon X\to \RX\colon x\mapsto f(\scal{a}{x})$,
and suppose that $T= P_F := (\Id+\partial F)^{-1}$ is the
associated proximity operator. 
Let $x\in X$. 
Then 
\begin{equation}
(\forall\nnn)\quad T^nx = P_{\{a\}^\perp}(x) +
P_f^n(\scal{a}{x})a,
\end{equation}
where $P_f := (\Id+\partial f)$ is the
proximity operator of $f$. 
Consequently, if $f$ is bounded below but without minimizers, then $T$
admits cosmic convergence and $(Q_n(x))_\nnn =
(T^nx/\|T^nx\|)_\nnn$ converges either to $+a$ or to $-a$.
\end{corollary}
\begin{proof}
Set $Y := \RR a$ and $\varphi\colon Y\to\RX\colon \xi a\mapsto
f(\xi)$. 
Then $F = \varphi \circ P_Y$ and $\partial F =
(\partial\varphi)\circ P_Y$. 
By \cref{t:extend}, $Tx = P_{\{a\}^\perp}(x) +
P_{\varphi}(\scal{a}{x}a) = P_{\{a\}^\perp}(x) +
P_{f}(\scal{a}{x})a$. 
Concerning the ``Consequently'' part, observe that if $f$ is
bounded below but without minimizers, then
$0\in\cran(\Id-P_f)\smallsetminus\cran(\Id-P_f)$ and the result
follows from \cref{t:1D}. 
\end{proof}

We conclude this section with two examples:
the first is covered by our analysis but the second is not. 

\begin{example}
Suppose that $X=\RR^2$ and set
\begin{equation}
F\colon \RR^2\to\RX\colon
(\xi_1,\xi_2)\mapsto 
\begin{cases}
\frac{1}{\xi_1+\xi_2}, &\text{if $\xi_1+\xi_2>0$;}\\
\pinf, &\text{otherwise,}
\end{cases}
\end{equation}
and suppose that $T=P_F$. 
Set $a:=(1,1)/\sqrt{2}$, and $f(\xi)=1/(\sqrt{2}\xi)$, if
$\xi>0$ and $f(\xi)=\pinf$ otherwise. Then \cref{c:extend}
applies and we obtain cosmic convergence; indeed, 
\begin{equation}
Q_n(x) = \frac{T^n(x)}{\|T^n(x)\|}\to a.
\end{equation}
\end{example}

\begin{example}
Suppose that $X=\RR^2$, set
\begin{equation}
F\colon \RR^2\to\RX\colon
(\xi_1,\xi_2)\mapsto 
\begin{cases}
\frac{\exp(\xi_1)}{\xi_2}, &\text{if $\xi_2>0$;}\\
\pinf, &\text{otherwise,}
\end{cases}
\end{equation}
and suppose that $T=P_F$. 
Then $F$ is \emph{not} of a form that makes \cref{c:extend}
applicable.
Interestingly, numerical experiments 
suggest that 
\begin{equation}
Q_n(x) = \frac{T^nx}{\|T^nx\|} \to (-1,0);
\end{equation}
however, we do not have a proof for this conjecture.
\end{example}

\subsection{Poincar\'e metric and cosmic interpretation}

\label{ss:cosmic}

In this section, we provide a different interpretation of our
convergence results which also motivates the terminology ``cosmic
convergence'' used above. 
We first observe that $X$ can be equipped with the 
\emph{Poincar\'e metric}, which is defined by
\begin{equation}
\Delta\colon X\to X\to\RR\colon (x,y) \mapsto  \left\| \frac{x}{1+\|x\|} - \frac{y}{1+\|y\|}
\right\|.
\end{equation}
Note that $\Delta$ is just the standard Euclidean metric after
the bijection $x\mapsto x/(1+\|x\|)$ between $X$ and the
\emph{open} unit ball was applied. 
The metric space $(X,\Delta)$ is not complete; however, regular
convergence of sequences in the Euclidean space $X$ is preserved. 
To complete $(X,\Delta)$, define the equivalence relation
\begin{equation}
x\equiv y 
\;\;:\Leftrightarrow\;\;
x\in\RPP y
\end{equation}
on $X\smallsetminus\{0\}$, with equivalence class
\begin{equation}
\dir x := \RPP x 
\end{equation}
for $x\in X\smallsetminus\{0\}$. 
Following \cite{Rock98}, we write 
\begin{equation}
\hzn X := \menge{\dir x}{x\in X\smallsetminus \{0\}}
\;\;\text{and}\;\;
\csm X := X \cup \hzn X.
\end{equation}
Here $\hzn$ is the \emph{horizon} of $X$ while $\csm X$
denotes the \emph{cosmic closure} of $X$. 
A convenient representer of $\dir x$ is $x/\|x\|$. 
These particular representers form the unit \emph{sphere}
which we can think of adjoining to the open unit ball. 
More precisely, we extend $\Delta$ from $X\times X$
to $\csm X\times \csm X$ as follows:
\begin{equation}
\left.\begin{matrix} x\in X\\ \dir y \in \hzn X\end{matrix}
\right\}
\;\;\Rightarrow\;\;
\Delta(x,\dir y) := \Delta(\dir y,x) := \left\| \frac{x}{1+\|x\|}
- \frac{y}{\|y\|}\right\|.
\end{equation}
and 
\begin{equation}
\left.\begin{matrix} \
\dir x\in \hzn X\\ \dir y \in \hzn X\end{matrix}
\right\}
\;\;\Rightarrow\;\;
\Delta(\dir x,\dir y) := \Delta(\dir y,\dir x) := \left\| \frac{x}{\|x\|}
- \frac{y}{\|y\|}\right\|.
\end{equation}
Equipped with $\Delta$, the Bolzano-Weierstrass theorem implies
that the cosmic closure $\csm X$ is a (sequentially) 
\emph{compact} metric space; in particular,
any sequence $(x_n)_\nnn$ in $X$ such that $\|x_n\|\to\pinf$ has
a convergent subsequence in $(\csm X,\Delta)$. 
In the previous sections, we concentrated on the case when 
$(x_n)_\nnn = (T^nx)_\nnn$ and $\Fix T=\varnothing$; then, of
course, it may or may not
be true that the entire sequence converges in $(\csm X,\Delta)$. 
This provides an \emph{a posteriori} motivation for our
terminology.

\subsection{Conclusion}

We have taken a closer look at Pazy's trichotomy
theorem for nonexpansive operators. The question whether or
not $(T^nx)_\nnn$ always cosmically converges when $T$ has no
fixed points remains open;
however, we have presented various partial results indicating that the
answer may be affirmative. 
Future work may focus on analyzing larger classes of nonexpansive
operators, e.g., general proximity operators or averaged
operators. 
Another promising avenue may be to use tools from non-euclidean
geometry (see \cite{GR}). 
Furthermore, it is presently unclear how the presented results extend
to infinite-dimensional settings. 


\end{document}